\documentclass[10pt]{article}
\usepackage{amsmath,amsfonts,amsthm,amssymb}
\usepackage{graphicx,xcolor, enumerate,tikz,caption,soul}
\usepackage{latexsym}
\usepackage{tabularx,longtable}
\usepackage{setspace}

\usepackage{soul}

\newtheorem{theorem}{Theorem}[section]

\newtheorem{lemma}[theorem]{Lemma}

\newtheorem{observation}[theorem]{Observation}

\theoremstyle{definition}
\newtheorem{definition}[theorem]{Definition}

\vspace{5cm}

\begin{document}

\textwidth4.5true in
\textheight7.2true in

\begin{title}
{\Large \bf {A note on nearly Platonic graphs\\ with connectivity one}}
\end{title}
\date{}
\begin{author}
{
	\sc{Dalibor Froncek}\\
       	Department of Mathematics and Statistics\\
  		University of Minnesota Duluth\\
		Duluth, Minnesota 55812\\
		U.S.A.\\
		{\texttt{dalibor@d.umn.edu}} \\ \\
	\sc{Mahdi Reza Khorsandi}\\
		Faculty of Mathematical Sciences\\ 
		Shahrood University of Technology\\
		P.O. Box 36199-95161, Shahrood\\ 
		Iran\\ 
		{\texttt{khorsandi@shahroodut.ac.ir}}\\ \\
	\sc{Seyed Reza Musawi}\\
		Faculty of Mathematical Sciences\\ 
		Shahrood University of Technology\\
		P.O. Box 36199-95161, Shahrood\\ 
		Iran\\ 
		{\texttt{r\_musawi@shahroodut.ac.ir}}\\ \\
	\sc{Jiangyi Qiu}\\
			Department of Mathematics and Statistics\\
			University of Minnesota Duluth\\
			Duluth, Minnesota 55812\\
			U.S.A.
			{\texttt{qiuxx284@d.umn.edu}} \\
}
\end{author}

\maketitle

\begin{abstract}
  A $k$-regular planar graph $G$ is nearly Platonic when all faces but one are of the same degree while the remaining face is of a different degree. We show that no such graphs with connectivity one can exist. This complements a recent result by Keith, Froncek, and Kreher on non-existence of $2$-connected nearly Platonic graphs.
\end{abstract}

\section{Introduction}

A \emph{Platonic graph of type $(k,d)$}  is a $k$-vertex regular and $d$-face regular planar graph. It is well known that there exist exactly five Platonic graphs, which can be viewed as skeletons of the five Platonic solids---tetrahedron, cube,  dodecahedron, octahedron, and icosahedron, of types $(3,3),(3,4),(3,5),(4,3)$ and $(5,3)$, respectively.

There are several classes of vertex-regular planar graphs with all but two faces of one degree and two faces of another degree. Hence, it is an intriguing question whether there  exist  vertex-regular planar graphs with exactly one exceptional face? This question was answered in the negative by Deza, Dutour Sikiri\v{c}, and Shtogrin~\cite{DezSikShtog} with a sketch of a proof, and for 2-connected graphs proved in detail by Keith, Froncek, and Kreher~\cite{KFK1}.

\begin{theorem}[\cite{DezSikShtog},~\cite{KFK1}]\label{thm:no-NPG}
  There is no finite, planar, {{$2$-connected}}, regular graph that has all but one face of one degree and a single face of a different degree.
\end{theorem}

We complement the result by offering a detailed case-by-case analysis for the remaining case with connectivity one. The main idea of our proof is the following.
 If such a graph with connectivity one exists, then there must exist an endblock, that is, a 2-connected graph with all vertices but one of degree $k$, one vertex of degree $1<l<k$, all faces but one of degree $d_1$ and one face of degree $d\neq d_1$. The non-existence of such graphs was  claimed by Deza and Dutour Sikiri\v{c} in~\cite{DezSik}. Because we were not satisfied with the proof, a purely combinatorial alternative is presented in this paper.

Our goal is to present an alternative proof of the following:

\begin{theorem}[\cite{DezSik}]\label{thm:main-conn=1}
  There is no finite, planar, regular graph with connectivity one that has all but one face of one degree and a single face of a different degree.
\end{theorem}

The main idea is to look at the blocks of such a potential graph and show that no endblock with required properties can exist.


\section{Basic notions and observations}

We start with a formal definition of an endblock.

\begin{definition}\label{def:block}
  A \emph{$(k,d_1,d)$-endblock $B(k,l)$} is a $2$-connected planar graph on $n$ vertices with $n-1$ vertices of degree $k$, one \emph{exceptional vertex} $x_1$ with $\deg(x_1)=l$ and $1< l<k$, all faces but one of \emph{{common} degree} $d_1$, and the remaining  face of degree $d\neq d_1$, where the exceptional vertex $x_1$ belongs to the face of degree $d$.
\end{definition}

We will often use in our arguments the notion of saturated paths.

\begin{definition}\label{def:sat-path}
  Let $G$ be a $2$-connected planar graph with maximum vertex degree $k$,  {common} face degree $d_1$ and outerface $x_1,x_2,\dots,x_{d}$ of degree $d\neq d_1$. 
   A vertex $u\neq x_1$ is  {\emph{$k$-saturated}} (or simply {\emph{saturated}}) if $\deg(u)=k$ and  {\emph{$k$-unsaturated}} (or simply {\emph{unsaturated}}) if $\deg(u)<k$. Similarly, for a given integer $2\le l<k$,  vertex $x_1$ is  {\emph{$l$-saturated}} (or simply {\emph{saturated}}) if $\deg(x_1)=l$ and  {\emph{$l$-unsaturated}} (or simply {\emph{unsaturated}}) if $\deg(x_1)<l$. 	\\
    Let a path $P=u_1,u_2,\dots,u_{d_1}$ be an induced subgraph of $G$ such that the graph $G+u_1u_{d_1}$ is still planar and the cycle $C=u_1,u_2,\dots,u_{d_1}$ is a boundary of a face of degree $d_1$.  If all vertices $u_i$ for $i=2,3,\dots, d_1-1$ are saturated and both $u_1$ and $u_{d_1}$ are unsaturated, then $P$  is called a {\emph{weakly-$k$-saturated $d_1$-path}}. If at most one of $u_1$ or $u_{d_1}$ is unsaturated while all other vertices are saturated, then $P$  is called a {\emph{strongly-$k$-saturated $d_1$-path}}. When $k$ and $d_1$ are fixed, we call these paths simply {\emph{weakly saturated}} or {\emph{strongly saturated}}, respectively.
\end{definition}
  
The following assertions are easy to verify.

\begin{observation}\label{obs:strong}
  Let $G$ be a $2$-connected planar graph with maximum vertex degree $k$,  minimum face degree $d_1$ and outerface $x_1,x_2,\dots,x_{d}$ of degree $d\neq d_1$. 	 
  If a strongly-$k$-saturated $d_1$-path $P=u_1,u_2,\dots,u_{d_1}$ is on a boundary of an  inner  face of $G$, then $G$ cannot be completed into a  $(k,d_1,d)$-endblock.
\end{observation}

\begin{proof} If $G$ is a subgraph of a $(k,d_1,d)$-endblock, then the whole path $P$ must be a part of an inner face of degree $d_1$, which implies that the remaining edge of that face must be $u_1u_{d_1}$. However, this edge cannot be added, because at least one of the degrees of $u_1$ and $u_{d_1}$ would then exceed $k$, a contradiction.
\end{proof}

\begin{observation}\label{obs:weak}
  Let $G$ be a $2$-connected planar graph with maximum vertex degree $k$,  minimum face degree $d_1$ and outerface $x_1,x_2,\dots,x_{d}$ of degree $d\neq d_1$.  If a weakly-$k$-saturated $d_1$-path $P=u_1,u_2,\dots,u_{d_1}$ is on a boundary of an  inner  face of $G$, then the edge $u_1u_{d_1}$ must be added in order to complete $G$ into a  $(k,d_1,d)$-endblock.
\end{observation}

\begin{proof}
  Similarly as above, the whole path $P$ must be a part of an inner face of degree $d_1$, which implies that the remaining edge of that face must be $u_1u_{d_1}$. Hence, we must add it to $G$ to complete it into the required endblock.
\end{proof}

\begin{observation}\label{obs:inside-triangle}
  Let $G$  be a subgraph of a $(k,3,d)$-endblock $B(k,l)$ and $u,v,w$ be a triangle such that $v$ is saturated  and has no neighbors inside the triangle. Then the triangle $u,v,w$ is a face boundary.
\end{observation}

\begin{proof} 
By Observation \ref{obs:weak}, the path $u,v,w$ must be a part of a triangular face. Suppose that $u$ has neighbors inside the triangle. Then at least one of them, say $u_1$, must be on the boundary containing edges $u_1u$ and $uv$. Since $v$ has no neighbors inside the triangle, the boundary also contains the edge $vw$. But then the edges $u_1u,uv,vw$ bound a face that is longer than a triangle, which is impossible.
\end{proof}

Now we start eliminating certain forbidden  configurations. In a $(k,d_1,d)$-endblock $B(k,l)$ with $x_1,x_2,\dots,x_{d}$ as the boundary of the exceptional face, by a \emph{chord} we mean an edge $x_ix_j$ not on the boundary of the exceptional face. That is, if $i<j$ and $x_ix_j$ is a chord, then $j-i\ne 1, d-1$.

\begin{lemma}\label{lem:no-chord-3}
  A  $(3,d_1,d)$-endblock $B(3,2)$ for $d_1=4,5$ does not have a chord.
\end{lemma}

\begin{proof}
  Let the cycle $x_1,x_2,\dots,x_{d}$ be the boundary of the exceptional face of this graph and there exists a chord $x_i x_j$ and $j>i$. Then $j-i\geq 3$, otherwise $j=i+2$ and $x_ix_{i+1}x_j$ is a triangle such that $x_i$ is saturated and has no neighbor inside the triangle. By Observation \ref{obs:inside-triangle}, this triangle is the boundary of a face, therefore, $x_{i+1}$ is of degree $2$, which is impossible.

  Now, we consider the subgraph $H$ induced by all vertices on and inside the cycle $x_i,x_{i+1},\dots, x_{j-1},x_j$. Create an isomorphic copy $\varphi(H)=H'$ of $H$ by assigning $\varphi(v)=v'$ for every $v\in H$. Then amalgamate the edges $x_ix_j$ and $x'_i x'_j$. The resulting graph is a 2-connected, $3$-regular planar graph with all faces of degree $d_1$ except the outerface, which is of degree $2(j-i)$. We proved above that $j-i\geq 3$, which implies that the outerface is of degree $2(j-i)\geq 6$. Thus we have constructed a 2-connected, 3-regular planar graph with one face of degree greater than $5$ and all remaining faces of degree $d_1\leq 5$. This contradicts Theorem \ref{thm:no-NPG}.
\end{proof}

\begin{lemma}\label{lem:no-chord-4}
  A $(4,3,d)$-endblock $B(4,l)$, with the cycle $x_1,x_2,\dots,x_{d}$ as the boundary of the exceptional face does not have a chord, other than $x_2x_{d}$ when $l=2$.
\end{lemma}

\begin{proof}
  Let the graph {have} some chords and the chord $x_i x_j$ with $j>i$ be the shortest one. It means that there is no other chord $x_sx_t$ with $0<t-s<j-i$.

If $i=1$, then $l=3$. In this case the graph has only one vertex of an odd degree, which is impossible.   

   {Now let $i>1$}
  and $y_i$ be the fourth neighbor of $x_i$. If $y_i$ is on or inside of the cycle $x_1,x_2,\dots,x_i,x_j,x_{j+1}\dots,x_{d}$, then the path $x_j,x_i,x_{i+1}$ is a weakly-4-saturated 3-path and we must have the triangular face $x_j,x_i,x_{i+1}$. If $j-(i+1)=1$, then $\deg(x_{i+1})=2$, which is impossible and so $j-(i+1)\ge2$ and $x_{i+1}x_j$ is a chord shorter than $x_ix_j$, a contradiction.
 Hence, $y_i$ must be inside of the cycle $x_i,x_{i+1},\dots,x_j$. By symmetry, the fourth neighbor $y_j$ of $x_j$ must be inside that cycle as well. But then we see that the path $x_{i-1},x_{i},x_{j}$ is a weakly saturated 3-path and by Observation \ref{obs:weak}, we must have the edge $x_{i-1}x_{j}$. 

If $j<d$, then $x_j$ has neighbors $x_{i-1}, x_i, y_j, x_{j-1}, x_{j+1}$ and is of degree at least 5, a nonsense.
Therefore, $j=d$ and we must have $x_{i-1}=x_1$, which concludes the proof. 
\end{proof}

\begin{lemma}\label{lem:no-chord-5}
  A  $(5,3,d)$-endblock $B(5,l)$, with the cycle $x_1,x_2,\dots,x_{d}$ as the boundary of the exceptional face does not have a chord  other than $x_2x_{d}$ when $l=2$. 
\end{lemma}

\begin{proof}
  Let the graph {have}  some chords and the chord $x_i x_j$ with $j-i>1$ be the shortest one. It means that there is no other chord $x_sx_t$ with $0<t-s<j-i$.

We denote by $C$ the cycle $x_i,x_{i+1},\dots,x_j$ and by $C'$ \\ the cycle $x_j,x_{j+1},\dots,x_{d},x_1,\dots,x_i$. 

First, we consider the case $i\ne1$ and so $\deg(x_i)=\deg(x_j)=5$. 
Call $y_i^1$ and $y_i^2$  the neighbors of $x_i$ other than $x_{i-1},x_{i+1},x_j$  and those of $x_j$ other than  $x_{j-1},x_{j+1}$(or $x_1$),$x_i$  similarly $y_j^1$ and $y_j^2$.
  
 We will discuss several cases based on placement of the vertices $y_s^t$ within cycles $C$ and $C'$.

  If both $y_i^1,y_i^2$ are within $C'$, then the path $x_j,x_i,x_{i+1}$ is a weakly-5-saturated 3-path and by Observations \ref{obs:weak} and \ref{obs:inside-triangle}, we must have the triangular face $x_i,x_{i+1},x_j$. Since $j-(i+1)<j-i$ the edge $x_{i+1}x_j$ is not a chord.  
  But then $x_{i+1}$ would have three other neighbors inside that face, which is impossible.
  
  Similarly to the previous case, if both $y_j^1,y_j^2$ are within $C'$, then the path $x_i,x_j,x_{j-1}$ is a weakly-5-saturated 3-path and by Observations \ref{obs:weak} and \ref{obs:inside-triangle}, we must have the triangular face $x_i,,x_jx_{j-1}$. Since $(j-1)-i<j-i$ the edge $x_{i}x_{j-1}$ is not a chord.  
  But then $x_{j-1}$ would have three other neighbors inside that face, which is impossible.

If one of $x_i, x_j$ has both remaining neighbors inside $C$, say $x_i$, then the path $x_{i-1},x_i,x_j$ is weakly 5-saturated path and we must have edge $x_{i-1}x_j$ completing the triangle. We observe that $x_{i-1}$ cannot have any neighbors inside this triangle. If $j=d$, then it follows that $ i-1=1$ and $l=2$ and we are done. If $j<d$, then by the previous case, $x_j$ has one neighbor other than $x_{j-1}$ inside $C$ and the graph bounded by the cycle $x_{i-1},x_i,\dots,x_j,x_{i-1}$ has $x_{i-1}$ of degree 2 and $x_j$ of degree 4. Hence, we can take two copies and amalgamate them to obtain a 5-regular graph with the outer face of degree more than 3 and all other faces triangular. However, such a graph does not exist, so this case is impossible.

  The only remaining case is that $x_i$ and $x_j$ have exactly one neighbor within both $C$ and $C'$. In this case, 
 we can again obtain a contradiction in a similar manner as in Lemma~\ref{lem:no-chord-4}. Denote by $H$ the induced subgraph of $G$ consisting of all vertices on or within $C$ and create an isomorphic copy $H'$. Then amalgamate $x_i$ with $x'_i$ and $x_j$ with $x'_j$. The resulting graph is a 2-connected 5-regular graph with the outerface of degree $2(j-i)>3$ and all other faces of degree 3. Such a graph cannot exist by Theorem \ref{thm:no-NPG}.
  
Finally, we consider the case $i=1$, that is, the graph has a chord $x_1x_j$ with $j\ne 2, d$.
 If $l=3$, then $x_1$ has no neighbor within $C$ and so by Observation \ref{obs:weak}, $x_2x_j$ is an edge and  the graph has a shorter chord than $x_1x_j$, which is impossible.
   
   For $l=4$, the vertex $x_1$ has the fourth neighbor $y_1^1\not\in\{x_2, x_3, \dots, x_j\}\cup\{x_{d}\}$.

   If  $y_1^1$ is not the inside of $C$, then as in the previous case, the graph has a {shorter}  chord $x_2x_j$, a contradiction. Thus, $y_1^1$ is within $C$. By applying  Observation \ref{obs:weak} on the weakly-4-saturated 3-path   $x_{d},x_1,x_j$, we deduce that the triangle $x_1,x_{d},x_j$ is the boundary of a triangular face of the graph.

We have $1<j<d$. If $x_jx_d$ is a chord, we find the shortest chord $x_{i'}x_{j'}$ such that $j\leq i'<j'\leq d$ and repeat
the case $i'\neq1$ from the first part of the proof.

{If} $x_jx_d$ is not a chord, by the first part of this proof, and so $j=d-1$ and $\deg(x_d)=2$, which is impossible.
  
  We have exhausted all possibilities and the proof is complete.
\end{proof}

\begin{lemma}\label{lem:t-values}
  Let $t$ be the number of vertices of the $(k,d_1,d)$-endblock $B(k,l)$ not on the boundary of the outer exceptional face. Then the values of $t$ are as follows:

  \centering
  $
  \begin{array}{
    |>{\centering\arraybackslash$} p{1.0cm} <{$}
    |>{\centering\arraybackslash$} p{1.0cm} <{$}
    |>{\centering\arraybackslash$} p{1.0cm} <{$}
    |>{\centering\arraybackslash$} p{1.8cm} <{$} |}
    \hline
    k & d_1 & l & t        \\
    \hline 
    3 & 3 & 2 & (5-d)/3 \\
    \hline
    3 & 4 & 2 & 3 \\
    \hline
    3 & 5 & 2 & d+7 \\
    \hline
    4 & 3 & 2 & 2 \\
    \hline
    5 & 3 & 2 & d+3 \\
    \hline
    5 & 3 & 3 & d+4 \\
    \hline
    5 & 3 & 4 & d+5 \\
    \hline
  \end{array}
  $
\end{lemma}

\begin{proof}
  Denote the order of the graph by $n$, the number of its edges by $m$ and the number of faces by $f$, thus the sum of  the vertex degrees will be $k(n-1)+l$, which is twice the number of edges. By Euler's formula, the number of faces is
   $$f = m+2-n = \frac{k(n-1)+l}{2}+2-n.$$ 
   Also since the sum of the face degrees is twice the number of edges, we have
    $$(f-1)d_1+d=2m = k(n-1)+l.$$ 
    Solve for $n$, we have 
    $$n=\frac{(2-d_1)(k-l)+2d_1+2d}{2k+2d_1-kd_1}.$$ 
    Recall that $t = n - d$, so when we plug in the corresponding values of $k, d_1$, and $l$, we  obtain our desired values of $t$ as a function of $d$.
\end{proof}

\section{Type $(3,d_1)$}

\begin{lemma}\label{lem:no-3-endblock-3}
  A  $(3,3,d)$-endblock $B(3,2)$ does not exist for any $d$.
\end{lemma}

\begin{proof}
  By Lemma \ref{lem:t-values}, we must have $d=2$ or $d=5$, otherwise $t$ is not a non-negative integer. Recall that the number of vertices is $d+t$. If $d=2$, then $t=1$ and the graph has 3 vertices in total. Hence, we cannot have vertices of degree 3. When $d=5$, then $t=0$ and the graph has 5 vertices in total. By applying  Observation \ref{obs:weak} on the weakly-2-saturated 3-path $x_5,x_1,x_2$ we conclude that $x_2x_5$ is an edge of the graph. Now, the path $x_5,x_2,x_3$ is a strongly-3-saturated 3-path. Hence, the graph has a face with the length greater than 3 and $G$ cannot be completed into $B(3,2)$.
\end{proof}

\begin{lemma}\label{lem:no-3-endblock-4}
  A  $(3,4,d)$-endblock $B(3,2)$ does not exist for any $d$.
\end{lemma}

\begin{proof}
  Recall that by $t$ we denote the number of vertices of $B(3,2)$ inside of the cycle bounding the face of degree $d$, that is, all vertices other than $x_1,x_2,\dots,x_{d}$. It follows from Lemma~\ref{lem:t-values} that $t=3$.
  
  We denote the internal vertices by $y_1, y_2$ and $y_3$. Since $d_1=4$, there are at most two edges $y_iy_j$, which implies that there are at least five edges $y_ix_j$. As there is no chord by Lemma~\ref{lem:no-chord-3},  each $x_i, i\neq 1$ has exactly one neighbor $y_j$  and hence $d\geq6$. Because $x_1$ is of degree 2, it is a saturated vertex. Let $x_2y_1$ be an edge. Then $y_1,x_2,x_1,x_{d}$ is a weakly-3-saturated 4-path, and we must have the edge $y_1x_{d}$.

  If the third neighbor of $x_3$ is $y_1$, then we have a triangular face, which is impossible. Assume that $x_3$ is adjacent to $y_2$. Then $y_1,x_2,x_3,y_2$ is a weakly-3-saturated 4-path, and we must have the edge $y_1y_2$. Now, the path $y_2,y_1,x_{d},x_{d-1}$ is a weakly-3-saturated 4-path, and we must have the edge {$y_2x_{d-1}$}. Since $d\ge6$, we have $d-1\ne 4$ and so $x_4\ne x_{d-1}$.
  Then $x_4,x_3,y_2,x_{d-1}$ is a strongly-3-saturated 4-path, and $B(3,2)$ cannot be completed.
\end{proof}

\begin{lemma}\label{lem:no-3-endblock-5}
  A  $(3,5,d)$-endblock $B(3,2)$ does not exist for any $d$.
\end{lemma}

\begin{proof}
Assume that the cycle $x_1,x_2,\dots,x_{d}$ is the boundary of the outerface and $\deg(x_1)=2$. By Lemma \ref{lem:no-chord-3}   the graph has no chord and so each $x_i$, except $x_1$, is adjacent to an interior vertex, say $y_{i-1}$. 
\begin{figure}[h]
\begin{tikzpicture}[scale=.8]
\draw (-0.001,-2.7)--(1.735,-2.069)--(2.658,-0.469)--(2.338,1.349)
--(0.923,2.537)--(-0.002,3)--(-0.924,2.537)--(-2.339,1.35)--
(-2.66,-0.469)--(-1.736,-2.069)--(-0.001,-2.7);
\filldraw(-0.002,3)circle(2pt);\node at (-0.002,3.3) {$x_1$};
\filldraw(-0.001,-2.7)circle(2pt);
\filldraw(1.735,-2.069)circle(2pt);
\filldraw(2.658,-0.469)circle(2pt);
\node at (2.95,-0.469) {$x_4$};
\filldraw(2.338,1.349)circle(2pt);
\node at (2.6,1.349) {$x_3$};
\filldraw(0.923,2.537)circle(2pt);
\node at (1,2.75) {$x_2$};
\filldraw(-0.924,2.537)circle(2pt);
\node at (-0.93,2.75) {$x_{d}$};
\filldraw(-2.339,1.35)circle(2pt);
\node at (-2.339,1.6){$x_{d-1}$};
\filldraw(-2.66,-0.469)circle(2pt);
\node at (-2.66,-0.65){$x_{d-2}$};
\filldraw(-1.736,-2.069)circle(2pt);
\node at (-1.736,-2.269){$x_{d-3}$};
\draw (-0.001,-2.7)--(-0.001,-2);	\draw(1.735,-2.069)--(1.285,-1.533);\draw(2.658,-0.469)--(1.969,-0.348);
\draw (2.338,1.349)--(1.732,0.999);
\draw (0.923,2.537)--(0.684,1.879);
\draw (-0.924,2.537)--(-0.684,1.879);
\draw (-2.339,1.35)--(-1.732,1);
\draw (-2.66,-0.469)--(-1.97,-0.348);
\draw (-1.736,-2.069)--(-1.286,-1.532);
\filldraw(-0.001,-2)circle(2pt);
\filldraw(1.285,-1.533)circle(2pt);
\filldraw(1.969,-0.348)circle(2pt);
\node at (1.65,-0.34) {$y_3$};
 \filldraw(1.732,0.999)circle(2pt);
 \node at (1.5,0.8) {$y_2$};
 \filldraw(0.684,1.879)circle(2pt);
 \node at (0.6,1.7) {$y_1$};
 \filldraw(-0.684,1.879)circle(2pt);
 \node at (-0.5,2.1) {$y_{d-1}$};
 \filldraw(-1.732,1)circle(2pt);
 \node at (-1.732,.8) {$y_{d-2}$};
 \filldraw(-1.97,-0.348)circle(2pt);
 \filldraw(-1.286,-1.532)circle(2pt);
 draw[green]-0.001,-2)--(-0.512,-1.41)--(-1.286,-1.532)--(-1.299,-0.752)--(-1.97,-0.348)--(-1.478,0.259)--(-1.732,1)--(-0.966,1.148)--(-0.684,1.879)--(0.684,1.879)--(0.963,1.149)--(1.732,0.999)--(1.476,0.261)--(1.969,-0.348)--(1.299,-0.749)--(1.285,-1.533)--(0.514,-1.41)--(-0.001,-2);
 \filldraw(0.514,-1.41)circle(2pt);
 \draw (0.514,-1.41)--(0.3,-1.1);
\filldraw(1.299,-0.749)circle(2pt);
\node at (1.55,-.9) {$z_{3}$};
\draw (1.299,-0.749)--(1,-.4);
\filldraw(1.476,0.261)circle(2pt);
\draw (1.476,0.261)--(0,0.6);
\draw (0,.6)--(0,0.9);
\draw (-1.478,0.259)-- (0,.6);
\filldraw[blue](0,.6)circle(2pt);
\node at (0,0.3) {$w_{2}$};
\node at (1.76,0.261) {$z_{2}$};
\filldraw(0.963,1.149)circle(2pt);
\draw (0.963,1.149)--(0,0.9);
\node at (1.2,1.3) {$z_{1}$};
\filldraw(-0.966,1.148)circle(2pt);
\draw[red] (-0.966,1.148)--(0,.9);
\node at (-0.966,1.4) {$z_{d-2}$};
\node at (0,1.2) {$w_{1}$};
\filldraw(0,0.9)circle(2pt);
\filldraw(-1.478,0.259)circle(2pt);
\node at (-1.478,0.1){$z_{d-3}$};
\filldraw(-1.299,-0.752)circle(2pt);
\draw (-1.299,-0.752)--(-1,-0.5);
\filldraw(-0.512,-1.41)circle(2pt);
\draw (-0.512,-1.41)--(-.3,-1);
\draw (7,-2.7)--(8.735,-2.069)--(9.658,-0.469)--(9.338,1.349)
--(7.923,2.537)--(7,3)--(-0.924+7,2.537)--(-2.339+7,1.35)--
(-2.66+7,-0.469)--(-1.736+7,-2.069)--(7,-2.7);
\filldraw(7,3)circle(2pt);
\filldraw(7,-2.7)circle(2pt);
\filldraw(8.735,-2.069)circle(2pt);
\node at (9,-2.35) {$x_{i+1}$};
\filldraw(9.658,-0.469)circle(2pt);
\node at (10,-0.469) {$x_i$};
\filldraw(9.338,1.349)circle(2pt);
\node at (9.9,1.349) {$x_{i-1}$};
\filldraw(7.923,2.537)circle(2pt);
\filldraw(-0.924+7,2.537)circle(2pt);
\node at (-0.93+7,2.75){$x_{j+1}$};
\filldraw(-2.339+7,1.35)circle(2pt);
\node at (-2.4+7,1.65) {$x_{j}$};
\filldraw(-2.66+7,-0.469)circle(2pt);
\node at (-2.66+7,-0.65){$x_{j-1}$};
\filldraw(-1.736+7,-2.069)circle(2pt);
\filldraw(-.3+7,-1)circle(2pt);
\node at (-.2+7,-.5) {$y$};
\draw (-0.512+7,-1.41)--(-.3+7,-1);
\draw (-2.339+7,1.35)--(-.3+7,-1)--(2.658+7,-0.469);
\draw[dashed] (-0.924+7,2.537)--(2.338+7,1.349);
\end{tikzpicture}
\caption{}
\label{fig1}
\end{figure}
All $y_i$'s are distinct, otherwise as we see in Figure \ref{fig1}(right), if $i<j$ and two vertices $x_i$ and $x_j$ have a common 
interior  neighbor, say $y$, then we have two cycles $x_{i},x_{i+1},\dots,x_{j-1},x_{j},y$ and
 $x_{1},x_{2},\dots,x_{i},x_{j},x_{j+1},\dots,x_{d}$ and we consider the cycle that the third neighbor of $y$ is not in. 
 If this cycle is a triangle then the graph has an interior triangular face or a cut-vertex, which is impossible. If this cycle is a square or pentagon then the graph has a cut-vertex, which is impossible. If the length of this cycle is greater than 5 then one of two paths $x_{i-1},x_{i},y,x_{j},x_{j+1}$ or $x_{i+1},x_{i},y,x_{j},x_{j-1}$ is a weakly-3-saturated 5-path and the graph has a chord $x_{i-1}x_{j+1}$ or  $x_{i+1}x_{j-1}$, respectively, which is impossible.

The path $y_{1},x_{2},x_1,x_{d},y_{d-1}$ is a weakly-3-saturated 5-path and so $y_1y_{d-1}$ is an edge.

For any $1\le i\le d-1$, the path $y_i,x_{i+1},x_{i+2},y_{i+1}$ is a weakly-3-saturated 4-path and so two vertices $y_i$ and $y_{i+1}$ have a common adjacent $z_i$ to construct a pentagonal face. These vertices are distinct. If $z_i=z_{i+1}$ for some $i$, then $\deg(y_{i+1})=2$ and if   $z_i=z_j$, where $j-i>1$, then $\deg(z_i)\ge4$, which is impossible.
If $d=3$, then $z_1$ is a cut-vertex, which is impossible and so $d\ge4$.

If $d\ge4$, then the path $z_{d-2},y_{d-1},y_1,z_1$ is a weakly-3-saturated 4-path and two vertices $z_{d-2}$ and $z_1$ must have a common neighbor $w_1$ to obtain a pentagonal face. We have $w_1\ne z_i$, $2\le i\le d-3$, otherwise  $\deg(w_1)\ge4$, which is impossible.
If $d=4$, then $w_1$ is a cut-vertex, which is impossible and so $d\ge5$.
If $d\ge5$, then the path $w_1,z_{1},y_{2},z_2$ is a weakly-3-saturated 4-path and two vertices $w_{1}$ and $z_2$ must have a common neighbor $w_2$ to obtain a pentagonal face (see Figure \ref{fig1} (left)). We have {$w_2\ne z_i$}, $3\le i\le d-4$, otherwise   $\deg(w_2)\ge4$, which is impossible.
If $d=5$, then $w_2$ is a cut-vertex and so $d\ge6$.
If $d=6$, then the cycle $w_2z_2y_3z_3$ is the boundary of a square face, which is impossible and so $d\ge7$.

If $d\ge7$, then the path $y_{d-3},z_{d-3},w_2,z_2,y_3$ is a strongly-3-saturated 5-path and we have an interior face with the length greater than 5, which is impossible.
\end{proof}

%

\section{Type $(4,3)$}

For the two remaining cases, we use dual graphs. Recall that the \emph{dual graph} $G^D$ of a planar graph $G$ with vertex, edge, and face sets $V(G),E(G), F(G)$, respectively,  has $V(G^D)=F(G),F(G^D)=V(G)$ and and edge $e=f_1f_2\in V(G^D)$ if and only if the faces $f_1$ and $f_2$ share an edge in $G$. In general, $G^D$ can be a multigraph with loops. In our case, we only look at dual graphs of blocks, hence no loops will arise. Concerning multiple edges, we can only have one double edge when $l=2$.

\begin{lemma}\label{lem:no-4-endblock}
  A  $(4,3,d)$-endblock $B(4,l)$ does not exist for any $l$ and $d$.
\end{lemma}

\begin{proof}
We cannot have $l=3$, as in that case the endblock would have a single vertex of an odd degree, a nonsense. Thus, we have $l=2$. 
  \begin{figure}[h]
 \begin{center}
\begin{tikzpicture}
\filldraw(0,4)circle(2pt);\node at(0,4.3){$x_1$};
\filldraw(1,3.5)circle(2pt);\node at(1.2,3.7){$x_2$};
\filldraw(1.6,2.8)circle(2pt);\node at(1.8,2.9){$x_3$};
\filldraw(2.2,2)circle(2pt);\node at(2.5,2){$x_4$};
\filldraw(-1.6,2.8)circle(2pt);\node at(-2.1,2.9){$x_{d-1}$};
\filldraw(-2.2,2)circle(2pt);\node at(-2.7,2){$x_{d-2}$};
\filldraw(-1,3.5)circle(2pt);\node at(-1.2,3.7){$x_d$};
\draw (-2.5,1.5)--(-2.2,2)--(-1.6,2.8)--(-1,3.5)--(0,4)--(1,3.5)--(1.6,2.8)--(2.2,2)--(2.5,1.5);

\draw[red] (-1.8,1.6)--(-2.2,2)--(-1.1,2.2)--(-1.6,2.8)--(0,2.8)--(-1,3.5)--(1,3.5)--(0,2.8)--(1.6,2.8)--(1.1,2)--(2.2,2)--(1.8,1.6);
\node at(0,3.7){$f_d=f_1$};\node at(1,3.1){$f_2$};
\node at(1.6,2.3){$f_3$};\node at(-1.6,2.3){$f_{d-2}$};
\node at(-.9,3.1){$f_{d-1}$};\node at(0,3.2){$g_d=g_2$};
\node at(1.1,2.5){$g_3$};\node at(1.7,1.8){$g_4$};
\node at(-.9,2.6){$g_{d-1}$};\node at(-1.5,1.9){$g_{d-2}$};
\end{tikzpicture}

 \end{center}
  \caption{$B(4,2)$}
 \label{fig-2}
\end{figure}
  
  First we show that $d\ge4$. Suppose $d=3$. Then the outerface is a triangle $x_1,x_2,x_3$. Vertex $x_1$ is saturated, hence the inner face containing $x_1$ is the triangle $x_1,x_2,x_3$. By Lemma~\ref{lem:t-values}, we have $t=2$, hence there are exactly two other vertices $y_1$ and $y_2$, each of degree 4. But then $y_1$ would have to be adjacent to $x_2,x_3,y_2$ and also to $x_1$, which is impossible.

 Now we denote the outerface of $B(4,2)$ by $f$ and an inner face containing edge $x_ix_{i+1}$ by $f_i$ for $i=1,2,\dots,d$. Further, for $i=2,3,\dots,d$ the inner face containing $x_i$ but not sharing an edge with the outerface will be denoted by $g_i$ (see Figure \ref{fig-2}). 

 Let $D$ be the dual graph of $B(4,2)$. Because $d\geq4$, we have $\deg_D(f)\geq4$.  
 Notice that we have double edge $ff_1$ (see Figure~\ref{fig-3}).
 
  \begin{figure}[h!]
 \begin{center}
\begin{tikzpicture}
\filldraw(0,4)circle(2pt);\node at(0,5.2){$f$};
\filldraw(0,4.8)circle(2pt);
\node at(-.7,4){$f_d=f_1$};

\filldraw(0,3.5)circle(2pt);\node at(0.08,3.3){$g_2=g_d$};
\filldraw(1,3.5)circle(2pt);\node at(.9,3.2){$f_2$};
\filldraw(1.6,2.8)circle(2pt);\node at(1.5,2.5){$g_3$};
\filldraw(2.2,2)circle(2pt);\node at(1.9,2){$f_3$};
\filldraw(-1,3.5)circle(2pt);\node at(-.85,3.1){$f_{d-1}$};
\filldraw(-1.6,2.8)circle(2pt);\node at(-1.2,2.6){$g_{d-1}$};
\filldraw(-2.2,2)circle(2pt);\node at(-1.7,2){$f_{d-2}$};

\draw (-2.5,1.5)--(-2.2,2)--(-1.6,2.8)--(-1,3.5)--(0,3.5)--(1,3.5)--(1.6,2.8)--(2.2,2)--(2.5,1.5);
\draw (0,3.5)--(0,4);


\draw (-1.6,2.8)--(-1.1,2.8);
\draw (1.6,2.8)--(1.1,2.8);
\draw (0,4.8)[out=-135, in=135]to(0,4);
\draw (0,4.8)[out=-45, in=45]to(0,4);
\draw (0,4.8) [out=-30, in=90] to(1,3.5);
\draw (0,4.8)[out=-150, in=90] to(-1,3.5);
\draw (0,4.8) [out=-15, in=90] to(2.2,2);
\draw (0,4.8) [out=-165, in=90] to(-2.2,2);
\draw (0,4.8) [out=0, in=90] to(2.6,1.8);
\draw (0,4.8) [out=-180, in=90] to(-2.6,1.8);
\node at (0,1.8){$D$};
\end{tikzpicture}

 \end{center}
 \caption{The dual graph of $B(4,2)$}
 \label{fig-3}
\end{figure}
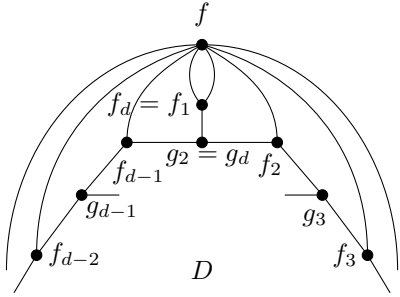

 Let $d=3q+r$ where $0\leq r\leq2$. We have $d\geq4$, hence $q\geq1$. We now construct a new graph $D'$ with $\Delta(D')=3$ as follows. We split vertex $f$ into vertices $f^0,f^1,\dots,f^q$ and each $f^i$ will be incident with edges $f^if_{3i+1},f^if_{3i+2},f^if_{3i+3}$ except possibly for $f^q$, which may be of degree zero, one, or two (see Figure \ref{fig-4}).

 \begin{figure}[h!]
 \begin{center}
\begin{tikzpicture}
\filldraw(1.5,4.1)circle(2pt);\node at(1.8,4.2){$f^0$};
\filldraw(3.5,3.1)circle(2pt);\node at(3.5,3.4){$f^1$};
\filldraw(-.7,4.1)circle(2pt);\node at(-1,4.2){$f^q$};
\filldraw(-2.3,3.6)circle(2pt);\node at(-2,3.9){$f^{q-1}$};
\filldraw(0,4)circle(2pt);
\node at(.1,4.3){$f_d=f_1$};
\filldraw(0,3.5)circle(2pt);\node at(0.1,3.3){$g_d=g_2$};
\filldraw(1,3.5)circle(2pt);\node at(.9,3.2){$f_2$};
\filldraw(1.6,3.2)circle(2pt);\node at(1.6,2.9){$g_3$};
\filldraw(2.2,2.8)circle(2pt);\node at(2.2,2.5){$f_3$};
\filldraw(2.8,2.5)circle(2pt);\node at(2.8,2.2){$g_4$};
\filldraw(3.4,2.2)circle(2pt);\node at(3.4,1.9){$f_4$};
\filldraw(-1,3.5)circle(2pt);\node at(-.9,3.2){$f_{d-1}$};
\filldraw(-1.6,3.2)circle(2pt);\node at(-1.5,2.9){$g_{d-1}$};
\filldraw(-2.2,2.8)circle(2pt);\node at(-2.2,2.5){$f_{d-2}$};
\filldraw(-2.8,2.5)circle(2pt);\node at(-2.8,2.2){$g_{d-2}$};
\filldraw(-3.4,2.2)circle(2pt);\node at(-3.4,1.9){$f_{d-3}$};
\draw (1.6,3.2)--(1.3,2.9);\draw (-1.6,3.2)--(-1.1,2.8);
\draw (2.8,2.5)--(2.5,2.2);\draw (-2.8,2.5)--(-2.6,2.2);

\draw (-3.6,1.9)--(-3.4,2.2)--(-2.8,2.5)--(-2.2,2.8)--(-1.6,3.2)--
(-1,3.5)--(0,3.5)
--(1,3.5)--(1.6,3.2)--(2.2,2.8)--(2.8,2.5)--(3.4,2.2)--(3.6,1.9);
\draw (0,3.5)--(0,4);

\draw (1.5,4.1)--(0,4);
\draw (1,3.5)--(1.5,4.1);
\draw (2.2,2.8)--(1.5,4.1);
\draw (3.5,3.1)--(3.4,2.2);
\draw (3.5,3.1)--(3.9,2);
\draw (-0.7,4.1)--(0,4);
\draw (-2.3,3.6)--(-1,3.5);
\draw (-2.3,3.6)--(-2.2,2.8);
\draw (-2.3,3.6)--(-3.4,2.2);
\node at (0,1.8){$D'$};
\end{tikzpicture}

 \end{center}
 \caption{$r=1$}
 \label{fig-4}
\end{figure}

 The outer boundary is now $f^0,f_3,g_4,f_4,f^1,f_6,\dots,f_1$. If $f^q$ is of degree zero, then we remove it and obtain an outerface $f^0,f_3,g_4,f_4,f^1,f_6,\dots,f_{d-2},f_{q-1},f_1$ of length at least six. But then we have a 2-connected, 3-vertex regular nearly Platonic graph with one face of size at least 6 and all other faces of size 4, which does not exist by Theorem \ref{thm:no-NPG}.
 
 
 If $f^q$ is of degree one, as in Figure \ref{fig-4}, we remove it and obtain an outerface  $f^0,f_3,g_4,f_4,f^1,f_6,\dots,f_{q-1},f_{d-1},g_2,f_1$ where $f_1$ is now of degree two  and we have a $(3,4,d')$-endblock $B(3,2)$, which does not exist by Lemma~\ref{lem:no-3-endblock-4}. 
 
 Finally, when  $f^q$ is of degree two, then the boundary is\\ $f^0,f_3,g_4,f_4,f^1,f_6,\dots,g_{d-1},f_{d-1},f_{q},f_1$ where $f^q$ is now of degree two and then again we have a  $(3,4,d')$-endblock $B(3,2)$, which does not exist by Lemma~\ref{lem:no-3-endblock-4}.
\end{proof}

\section{Type $(5,3)$}	

\begin{lemma}\label{lem:no-5-endblock}
	A  $(5,3,d)$-endblock $B(5,l)$ does not exist for any $l$ and $d$.
\end{lemma}

\begin{proof}
  We again use the dual graph technique to prove the claim. We start with an observation that  the case $d=3$ is impossible. If we have such a graph $G$ with vertex $x_1$ of degree $l$ where $3\leq l\leq4$, all other vertices of degree 5, inner faces triangular, and $d=3$, then the outerface is also a triangle. But then the dual graph $G^D$ is a 2-connected, cubic graph with one face of degree $l\neq5$, and all remaining faces of degree 5, which is impossible by Theorem~\ref{thm:no-NPG}.
  
  If $l=2$, then the path $x_3,x_1,x_2$ is weakly saturated, and we must have the edge $x_2x_3$ completing the boundary of the inner triangular face. But then the remaining neighbors of $x_2$ and $x_3$ are outside of the cycle $x_1,x_2,x_3$, that is, within the outerface, which is impossible.
  
  We use the same notation as in the previous proof, with the exception that  for $i=1,2,\dots,d$ the two inner faces containing $x_i$ but not sharing an edge with the outerface will be denoted by $g_i$ and $h_i$ (see Figure \ref{fig5}). 
  
  The case $l=2$ is essentially the same as for the $(4,3,d)$-endblock $B(4,2)$ and we omit it.
  

    When $l=3$, the graphs $B(5, 3)$ and its dual graph are shown in Figures \ref{fig5} and \ref{fig6}, respectively. After splitting vertex $f$ the outerface of $D'$ is  $f^0,f_3,g_4,h_4,f_4,f^1,\dots,f_1$ of length at least six. 
   
  \begin{figure}[h]
    \begin{center}
    \begin{tikzpicture}
  \filldraw(-2,1.5)circle(2pt);\node at(-2,1.8){$x_{d-1}$};
   \filldraw(-1,1.9)circle(2pt);\node at(-1,2.1){$x_{d}$};
    \filldraw(0,2.4)circle(2pt);\node at(0,2.7){$x_{1}$};
     \filldraw(1,1.9)circle(2pt);\node at(1,2.1){$x_{2}$};    
      \filldraw(2,1.5)circle(2pt);\node at(2,1.8){$x_{3}$};
\node at(.3,2){$f_{1}$};   \node at(1.4,1.5){$f_{2}$};   
\node at(-.3,2){$f_{d}$};   \node at(-1.4,1.5){$f_{d-1}$};        
\node at(.6,1.6){$g_{2}$};   \node at(1,1.3){$h_{2}$};      
\node at(1.6,1.1){$g_{3}$};   \node at(1.9,.9){$h_{3}$}; 
\node at(-.6,1.6){$h_{d}$};   \node at(-.95,1.2){$g_{d}$};        
\draw (-2.3,1.2)--(-2,1.5)--(-1,1.9)--(0,2.4)--(1,1.9)--(2,1.5)--(2.3,1.2);
\draw (-2.1,.9)--(-2,1.5)--(-1.7,.9);\draw (-1.2,1.1)--(-2,1.5);
\draw (-1.2,1.1)--(-1,1.9)--(-.7,1.1);\draw (0,1.7)--(-1,1.9);
\draw (0,1.7)--(0,2.4);
\draw (1.2,1.1)--(1,1.9)--(.7,1.1);\draw (0,1.7)--(1,1.9);
\draw (2.1,.9)--(2,1.5)--(1.7,.9);\draw (1.2,1.1)--(2,1.5);
\node at (0,1){$B(5,3)$};
  \end{tikzpicture}
  \end{center}
  \caption{$l=3$}
  \label{fig5}
  \end{figure}

    \begin{figure}[h]
    \begin{center}
    \begin{tikzpicture}
 \filldraw(-3,1.4)circle(2pt);\node at(-3,1.7){$h_{d-1}$};\filldraw(-2.4,1.6)circle(2pt);\node at(-2.4,1.3){$f_{d-1}$};\filldraw(-1.8,1.8)circle(2pt);\node at(-1.8,2.1){$g_{d}$};       
  \filldraw(-1.2,2)circle(2pt);\node at(-1.2,2.3){$h_{d}$};
   \filldraw(-.6,2.2)circle(2pt);\node at(-.6,1.9){$f_{d}$};
    \filldraw(0,2.6)circle(2pt);\node at(0,2.9){$f$};
   \filldraw(.6,2.2)circle(2pt);\node at(.6,1.9){$f_{1}$};    
  \filldraw(1.2,2)circle(2pt);\node at(1.2,2.3){$g_{2}$};    
   \filldraw(1.8,1.8)circle(2pt);\node at(1.8,2.1){$h_{2}$};
 \filldraw(2.4,1.6)circle(2pt);\node at(2.4,1.3){$f_{2}$};  
 \filldraw(3,1.4)circle(2pt);\node at(3,1.7){$g_{3}$};

\draw (-3.2,1.2)--(-3,1.4)--(-2.4,1.6)--(-1.8,1.8)--(-1.2,2)--(-.6,2.2)
--(0,2.6)--(.6,2.2)--(1.2,2)--(1.8,1.8)--(2.4,1.6)--(3,1.4)--(3.2,1.2);
\draw (.6,2.2)--(-.6,2.2);
\draw (0,2.6)[out=-170,in=80]to (-2.4,1.6);\draw (0,2.6)[out=-10,in=100]to (2.4,1.6);
\draw (0,2.6)[out=180,in=90]to (-3.5,1.6);\draw (0,2.6)[out=0,in=90]to (3.5,1.6);
\draw (-3,1.4)--(-3,1.1);\draw (-1.8,1.8)--(-1.7,1.5);
\draw (-1.2,2)--(-1.2,1.7);\draw (1.2,2)--(1.2,1.9);
\draw (3,1.4)--(3,1.1);\draw (1.8,1.8)--(1.7,1.5);

\node at (0,1){$D$};
  \end{tikzpicture}
  \end{center}
  \caption{The dual graph of $B(5,3)$}
  \label{fig6}
  \end{figure}
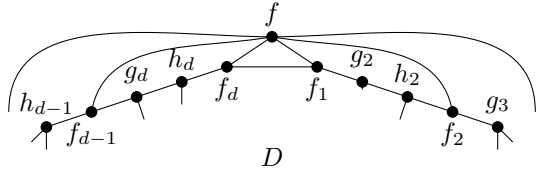
  
When $f_q$ is of degree zero,  the boundary is
   $f^0,f_3,g_4,h_4,f_4,f^1,\dots,f^{q-1},f_d,f_1$ 
   and we have a 2-connected, 3-vertex regular 
    nearly Platonic graph with one face of size at least 7 and all other faces of size 5, and such a graph does not exist by Theorem~\ref{thm:no-NPG}.  When $f^q$ is of degree one, by omitting $f^q$, the boundary is $f^0,f_3,g_4,h_4,f_4,f^1,\dots,g_d,h_d,f_d,f_1$  
   and $\deg(f_d)=2$, then the outer boundary is of length at least 6, and when $f^q$ is of degree two, the boundary is $f^0,f_3,g_4,h_4,f_4,f^1,\dots,f_{d-1},f^q,f_d,f_1$  and $\deg(f^q)=2$, then the outer boundary is of length at least 7. In the both cases, we have a $(5,3,d')$-endblock $B(5,3)$, which does not exist by Lemma~\ref{lem:no-3-endblock-4}.
  
  When $l=4$, then the outerface in $D'$ is  $f^0,f_3,g_4,h_4,f_4,f^1,\dots,g_1,f_1$ of length at least seven.
  
    \begin{figure}[h]
    \begin{center}
    \begin{tikzpicture}
    \draw (-5.5,1)--(5.5,1);
  \filldraw(-4.8,1)circle(2pt);\node at(-4.8,1.3){$h_{d-2}$};  
  \filldraw(-4.2,1)circle(2pt);\node at(-4.2,.7){$f_{d-1}$};     
     \filldraw(-3.6,1)circle(2pt);\node at(-3.6,1.3){$g_{d-1}$};   
 \filldraw(-3,1)circle(2pt);\node at(-3,1.3){$h_{d-1}$};\filldraw(-2.4,1)circle(2pt);\node at(-2.4,.7){$f_{d-1}$};\filldraw(-1.8,1)circle(2pt);\node at(-1.8,1.3){$g_{d}$};       
  \filldraw(-1.2,1)circle(2pt);\node at(-1.2,1.3){$h_{d}$};
   \filldraw(-.6,1)circle(2pt);\node at(-.6,.7){$f_{d}$};
    \filldraw(0,1)circle(2pt);\node at(0,1.3){$g_1$};
   \filldraw(.6,1)circle(2pt);\node at(.6,.7){$f_{1}$};    
  \filldraw(1.2,1)circle(2pt);\node at(1.2,1.3){$g_{2}$};    
   \filldraw(1.8,1)circle(2pt);\node at(1.8,1.3){$h_{2}$};
 \filldraw(2.4,1)circle(2pt);\node at(2.4,.7){$f_{2}$};  
 \filldraw(3,1)circle(2pt);\node at(3,1.3){$g_{3}$};   
      \filldraw(3.6,1)circle(2pt);\node at(3.6,1.3){$h_{3}$};   
   \filldraw(4.2,1)circle(2pt);\node at(4.2,.7){$f_{3}$};  
   \filldraw(4.8,1)circle(2pt);\node at(4.8,1.3){$g_{4}$};  
 
  \filldraw(-4.8,2)circle(2pt);\node at(-4.8,2.3){$f^{q-1}$};  
 \filldraw(2.4,2)circle(2pt);\node at(2.4,2.3){$f^{0}$};  
  \filldraw(-1.5,2)circle(2pt);\node at(-1.5,2.3){$f^{q}$};  
\draw (-2.4,1)--(-1.5,2)--(-.6,1);
\draw (2.4,2)--(2.4,1);  
\draw (.6,1)--(2.4,2)--(4.2,1);
 \draw (-4.8,2)--(-4.2,1);  
\draw (-5.3,1.5)--(-4.8,2)--(-5.6,1.6);

\draw (4.8,1)--(4.8,.5);\draw (3.6,1)--(3.6,.5);
\draw (3,1)--(3,.5);\draw (1.8,1)--(1.8,.5);
\draw (1.2,1)--(1.2,.5);
\draw (0,1)--(0,.5);
\draw (-4.8,1)--(-4.8,.5);\draw (-3.6,1)--(-3.6,.5);
\draw (-3,1)--(-3,.5);\draw (-1.8,1)--(-1.8,.5);
\draw (-1.2,1)--(-1.2,.5);

\node at (0,0){$q\ge2, r=2$};
  \end{tikzpicture}
  \end{center}
  \caption{ $D'$}
  \label{fig7}
  \end{figure}

  Now similarly as in the previous proof, when $f_q$ is of degree zero,  the boundary is
   $f^0,f_3,g_4,h_4,f_4,f^1,\dots,f^q,f_d,g_1,f_1$ 
   and we have a 2-connected, 3-vertex regular 
    nearly Platonic graph with one face of size at least 8 and all other faces of size 5, and such a graph does not exist by Theorem~\ref{thm:no-NPG}.  When $f^q$ is of degree one, by omitting $f^q$, the boundary is $f^0,f_3,g_4,h_4,f_4,f^1,\dots,g_d,h_d,f_d,g_1,f_1$ and $\deg(f_d)=2$, then the outer boundary is of length at least 7, and when $f^q$ is of degree two, the boundary is $f^0,f_3,g_4,h_4,f_4,f^1,\dots,f^q,f_d,g_1,f_1$ and $\deg(f^q)=2$, then the outer boundary is of length at least 9 (see Figure \ref{fig7}). In the both cases, we have a $(5,4,d')$-endblock $B(5,4)$, which does not exist by Lemma~\ref{lem:no-3-endblock-4}.
\end{proof}

\section{Main result}

Now we are ready to prove our main result.

\begin{theorem}\label{thm:no-endblock}
  There is no $(k,d
  _1,d)$-endblock for any admissible triple $(k,d_1,d)$.
\end{theorem}

\begin{proof}
  Follows directly from Lemmas~\ref{lem:no-3-endblock-3}--\ref{lem:no-3-endblock-5},~\ref{lem:no-4-endblock}, and~\ref{lem:no-5-endblock}.
\end{proof}

An alternative proof of the result presented by Deza and Dutour Sikiri\v c in~\cite{DezSik} now follows immediately.

\begin{theorem}
  There is no finite, planar, regular graph with connectivity one that has all but one face of one degree and a single face of a different degree.
\end{theorem}

\begin{proof}
  It is well known that every graph with connectivity one and minimum degree at least three has at least two endblocks, that is, 2-connected graphs with minimum degree at least two. If there was a graph defined in the Theorem, 
  it would have to contain a $(k,d_1,d)$-endblock for some admissible triple $(k,d_1,d)$. However, such endblock does not exist by Theorem~\ref{thm:no-endblock}. This proves the claim.
\end{proof}

\end{document}